\documentclass[12pt]{article}
\usepackage{amsthm, amsmath, amssymb}
\usepackage{tikz}
\usepackage{fullpage}
\usepackage{pgfplots}
\pgfplotsset{compat=newest}
\usepackage{float}
\usepackage{hyperref}

\newtheorem{thm}{Theorem}
\newtheorem{cor}{Corollary}
\newtheorem{lem}{Lemma}
\newtheorem{prop}{Proposition}
\newtheorem{con}[thm]{Conjecture}

\newtheorem{ques}{Question}

\begin{document}
\title{Symmetry and asymmetry between positive and negative square energies of graphs}
\author{Clive Elphick\thanks{\texttt{clive.elphick@gmail.com}, 
    School of Mathematics, University of Birmingham, Birmingham, UK.}
    \quad 
    William Linz\thanks{\texttt{wlinz@mailbox.sc.edu}, 
    Department of Mathematics, University of South Carolina, Columbia, South Carolina, 29208, USA. Partially supported by NSF RTG Grant DMS 2038080.}
} 

\maketitle
\abstract{The positive and negative square energies of a graph, $s^+(G)$ and $s^-(G)$, are the sums of squares of the positive and negative eigenvalues of the adjacency matrix, respectively. The first results on square energies revealed symmetry between $s^+(G)$ and $s^-(G)$. This paper reviews examples of asymmetry between these parameters, for example using large random graphs and the ratios $s^+/s^-$ and $s^-/s^+$, as well as new examples of symmetry. We answer some questions previously asked about $s^{+}$ and $s^{-}$ and suggest several further avenues of research.}

\begin{center}\emph{In Memoriam: Prof. Vladimir Nikiforov.}\end{center}

\section{Introduction}

Let $G$ be  a connected graph with $n$ vertices and $m$ edges and let $A$ denote the adjacency matrix of $G$, where $\mu_1 \ge  \ldots \ge \mu_n$ denote the eigenvalues of $A$.  Let $\chi(G)$ denote the chromatic number of $G$ and let:

\[
s^+ = \sum_{\mu_i > 0} \mu_i^2  \mbox{       and        } s^- = \sum_{\mu_i < 0} \mu_i^2.
\]

 Note that:

\[
\sum_{i=1}^n \mu_i^2 = tr(A^2) = 2m = s^+ + s^-.
\]

 Elphick \emph{et al} \cite{elphick16} formulated the following conjecture, which was placed $1^{st}$ in a wide ranging review of unsolved problems in spectral graph theory by Liu and Ning \cite{liu}:
 
\begin{con}\label{con:elphick}
For any connected graph:

\[
\min{(s^+ , s^-)} \ge n - 1.
\]

\end{con}

Elphick \emph{et al.} proved this result for various classes of graphs, including regular graphs. 

Note that for trees, $s^+ = s^- = m = n - 1$ and that for complete graphs $s^- = n - 1$. This conjecture provides an example of symmetry between $s^+(G)$ and $s^-(G)$, in that the same lower bound is tight for both $s^{+}$ and $s^{-}$. Abiad \emph{et al} \cite{abiad} gave the name positive and negative square energies to $s^+$ and $s^-$, and made some progress toward proving Conjecture~\ref{con:elphick} by developing new tools. 

As another example of symmetry, Ando and Lin \cite{ando} proved a conjecture due to Wocjan and Elphick \cite{wocjan13} that:

\[
1 + \max{\left(\frac{s^+(G)}{s^-(G)} , \frac{s^-(G)}{s^+(G)}\right)} \le \chi(G). 
\]

Coutinho and Spier \cite{coutinho} have recently strengthened this bound by proving a conjecture due to Wocjan, Elphick and Anekstein \cite{wocjan23} that:

\[
1 + \max{\left(\frac{s^+(G)}{s^-(G)} , \frac{s^-(G)}{s^+(G)}\right)} \le \chi_v(G)\le \chi(G),
\]

where $\chi_v(G)$ denotes the vector chromatic number.

 Because Conjecture~\ref{con:elphick} seems surprisingly difficult, it is useful to develop a more thorough understanding of positive and negative square energies. In this paper, we answer some questions asked by Abiad \textit{et al.}~\cite{abiad}, suggest several new avenues for research on $s^{+}$ and $s^{-}$, and make some initial progress on addressing several of these avenues. We have organised our paper in survey form, centering around different examples demonstrating symmetry or asymmetry between $s^{+}$ and $s^{-}$. By asymmetry, we mean any inequality, statement or expression which holds for $s^{+}$, but not $s^{-}$, or vice-versa. Typically, asymmetry results from contributions of the square of the principal eigenvalue $\mu_1$ to $s^{+}$. We hope that the questions and results we present here will spur further study into positive and negative square energies.  
 
 Sections 2 to 6 investigate examples of asymmetry between $s^+$ and $s^-$. In Section 2, we determine $s^{+}$ and $s^{-}$ for almost all graphs, addressing a question asked by Abiad \textit{et al.}~\cite{abiad}. Our result is the following. 

 \begin{thm}\label{thm:almostall}
Let $G$ be a graph on $n$ vertices. Then, with probability tending to $1$, 
\[s^+(G) = \left(\frac38 + o(1)\right)n^2\]
and 
\[s^-(G) = \left(\frac18 + o(1)\right)n^2.\]
\end{thm}
 
 In Sections 3 and 4, we investigate graphs for which $s^{-}$ is much larger than $s^{+}$. In Section 3, we study the ratios of square energies $s^{+}/s^{-}$ and $s^{-}/s^{+}$ and construct examples of generalised quadrangles where the latter ratio is large. In Section 4, we introduce and study squared spreads of graphs $s^{+} - s^{-}$, in analogy to the well-studied notion of the \textit{spread} of a graph~\cite{GHK2001}. In particular, we construct examples of graphs with very negative squared spread. 
 
 In Section 5, we consider $s^{+}$ and $s^{-}$ for maximal planar graphs. In Section 6, we consider upper bounds for $s^-$ that are not upper bounds for $s^+$. 
 
 Sections 7 to 9 provide further examples of symmetry between $s^+$ and $s^-$. In Section 7, we identify an infinite number of nonbipartite graphs with $s^+ = s^-$, answering a question of Abiad \textit{et al.}~\cite{abiad}. Our constructions are certain infinite families of Kneser graphs. 

 \begin{thm}\label{thm:knes1}
 Let $k \ge 2$. Then, the Kneser graph $K(2k+2j, k)$ has $s^{+} = s^{-}$ for every $1\le j < k$. 
 \end{thm}
 
 In Section 8, we prove a weaker lower bound than Conjecture~\ref{con:elphick}, and in Section 9, we consider disconnected graphs and propose a weaker conjecture than Conjecture~\ref{con:elphick}, but for all graphs. We conclude in Section 10 and make some remarks on Conjecture~\ref{con:elphick}.

\section{Average square energies}

\subsection{Introduction}

Theorem 2.5 in \cite{abiad} proves that if $H = G - e$ (where $e$ is an edge in $G$), then

\[
s^+(G) \ge s^+(H) - \theta_2^2 \mbox{  and  } s^-(G) \ge s^-(H) - \theta_n^2,
\]

where $\theta_1 \ge  ... \ge \theta_n$ are the eigenvalues of $A(H)$.

However, an increase in $s^+$ due to edge deletion is rare, as discussed in Section 2.2 of \cite{abiad}, whilst an increase in $s^-$ due to edge deletion is common, because connected graphs with the minimum number of edges (trees) and with the maximum number of edges ($K_n$) both have $s^- = n - 1$. Note that $s^-$ is maximised for regular complete bipartite graphs, with $s^- = m = n^2/4$ \cite{elphick16}, while $s^+ = (n - 1)^2$ is maximised for complete graphs. This provides an example of asymmetry in square energies. 

This implies that as the number of edges increases,  $s^-$ initially increases but reaches a maximum and then decreases. To quantify this intuition, we define  a new parameter $s^-(n,m)$ to be the average value of $s^-$ for all non-isomorphic  graphs with $n$ vertices and $m$ edges, and define $s^+(n,m)$ similarly. These parameters are termed average square energies. Clearly:

\[
s^+(n,m) + s^-(n,m) = 2m.
\]

We computed $s^+(n,m)$ and $s^-(n,m)$ for $n \le 9$. For example with $n = 9$, $s^+(n,m)$ monotonically increases with $m$ but $s^-(n,m)$ monotonically increases until $m = 24$ and then monotonically decreases.

\subsection{Average square energies for random graphs}

We computed average square energies for samples of random graphs using the Wolfram Mathematica function RandomGraph[n,p], with $n = 100$. The results are in Figure 1, with $s^-$ in red and $s^+$ in blue. It can be seen that $s^+$ monotonically increases with increasing $p$ while $s^-$ reaches a maximum value of $\approx 0.14n^2$ at  $p \approx 0.5$. 

\begin{figure}[ht]
\centering
\begin{tikzpicture}

\begin{axis}[xmin=0, xmax = 1, ymin=0, ymax= 10000, xlabel = {$p$}, ylabel = {average square energies}]

\addplot+[color=red] coordinates{
(0, 0)
(0.1, 464)
(0.2, 848)
(0.3, 1149)
(0.4, 1337)
(0.5, 1437)
(0.6, 1426)
(0.7, 1308)
(0.8, 1080)
(0.9, 701)
(1.0, 99)
};

\addplot+[color=blue] coordinates{
(0, 0)
(0.1, 526)
(0.2, 1132)
(0.3, 1821)
(0.4, 2623)
(0.5, 3513)
(0.6, 4514)
(0.7, 5622)
(0.8, 6840)
(0.9, 8209)
(1.0, 9801)
};

\end{axis}

\end{tikzpicture}

\caption{Average square energies for random graphs with $n = 100$ }
\end{figure}
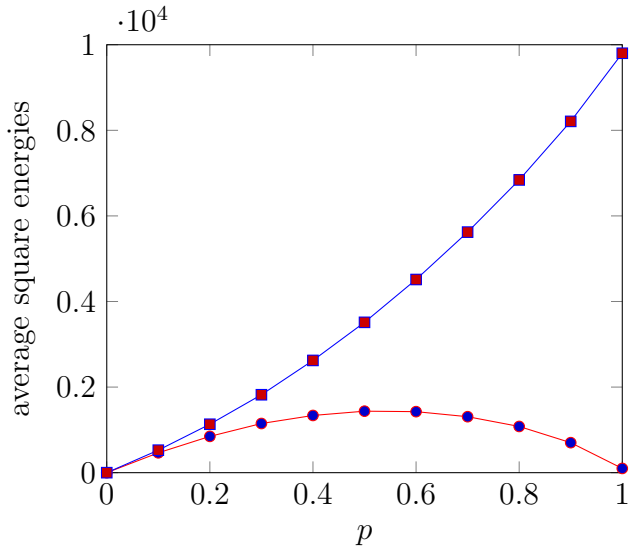

To prove results about average square energies we use the Erd\H{o}s-R\'{e}nyi model $G(n,p)$ where $n$ is the number of vertices and $p$ is the independent probability of each edge being present. 

Nikiforov \cite{nikiforov} investigated Schatten \emph{p}-norms of random graphs, and proved that these norms behave differently for $p < 2$, $p = 2$ and $p > 2$. Square energies correspond to $p = 2$. Nikiforov established that Schatten \emph{p}-norms of almost all graphs can be reduced to finding the Schatten \emph{p}-norms of the random graph $G(n, 1/2)$. We adapt Nikiforov's approach and notation in the following analysis, with a.s. meaning ``with probability tending to 1".

We determine the positive and negative square energies of almost all graphs by finding $s^{+}(G(n, \frac12))$ and $s^{-}(G(n, \frac12))$ a.s. Our argument follows a similar line of argumentation as  Nikiforov~\cite[Theorem 5]{nikiforov} for determining the Schatten $p$-norm of almost all graphs and determining the energy~\cite{Nik2} of almost all graphs. 

\begin{proof}[Proof of Theorem~\ref{thm:almostall}]
First, observe that since $\sum_{i=1}^n\mu_i^2 =2m$, it follows that a.s. 

\begin{equation}\label{eqn:sumsqs}
\sum_{i=1}^n \mu_i^2(G(n,\frac12)) = \left(\frac12 +o(1)\right)n^2.
\end{equation}

It is standard that the adjacency matrix $A(G(n, \frac12))$ is a random symmetric matrix with zero diagonal entries and independent entries $a_{ij}$ satisfying $\mathbb{E}[a_{ij}] = \frac12$, $\text{Var}[a_{ij}^2] = \frac14$. Then, a result of F\"uredi and Koml\'os~\cite[Theorem 1]{FK1981} implies

\begin{equation}\label{eqn:mu1mu2}
 \mu_1^2 = \left(\frac14 + o(1)\right)n^2 \mbox{  a.s.  }  \mbox{  and  } 
 \mu_2^2 \le \left(1+o(1)\right)n \mbox{ a.s. }
\end{equation}

Using Wigner's \cite{wigner} semicircle law in the form given by Arnold~\cite{arnold} (with appropriate normalization):

\[
s^- = \frac{2n^2}{\pi} \int_{-1}^0  x^2 \sqrt{1 - x^2}\, dx ;
\]

\[
s^+ - \mu_1^2 = \frac{2n^2}{\pi} \int_0^1 x^2 \sqrt{1 - x^2}\, dx ;
\]

 These integrals are both $\pi/16$, so a.s. $s^- = s^+ - \mu_1^2$. Therefore, using the formulae in \eqref{eqn:sumsqs} and \eqref{eqn:mu1mu2}, 

\[
s^+ = \frac{1}{2}\left(\frac34 + o(1)\right)n^2 \mbox{  a.s.  and  } s^- = \frac{1}{2}\left(\frac14 + o(1)\right)n^2 \mbox{  a.s.}
\]

\end{proof}

 This asymmetry between $s^+$ and $s^-$ for large random graphs is due to the largest eigenvalue. 

Question 4.1 in \cite{abiad} asks whether the percentage of graphs with $s^- > s^+$ tends to zero as $n \rightarrow \infty$? Theorem~\ref{thm:almostall} shows the answer is yes.

It would be interesting to calculate the expected value of the positive and negative square energies of the random graph $G(n, p)$. We believe that similarly to Theorem~\ref{thm:almostall}, 
\[
s^+ = \frac{1}{2}(p(p + 1) + o(1))n^2 \mbox{  a.s.  and  } s^- = \frac{1}{2}(p(1 - p) + o(1))n^2 \mbox{  a.s.}
\]

 \section{Ratios between positive and negative square energies}
 How large can the ratios $s^{+}/s^{-}$ and $s^{-}/s^{+}$ be? For complete graphs, $s^{+}/s^{-} = n-1$. 

 It is challenging to construct a family of graphs on $n$  vertices such that $s^{-}/s^{+} \ge a_n$ where $a_n \rightarrow \infty$ as $n\rightarrow \infty$. We present such 
 constructions using generalised quadrangles. 

 \begin{thm}\label{thm:gqex}
     There is a family of graphs on $n$ vertices with $s^{-} / s^{+} = \Theta(n^{\frac 14})$ as $n\rightarrow \infty$. 
 \end{thm}

 \begin{proof}
The collinearity graph of a generalised quadrangle of order $(s, t)$ is a strongly regular graph (see~\cite[Proposition 2.2.18]{BvM}) with $(s+1)(st+1)$ vertices and spectrum $k^1r^fa^g$, where \[k = s(t+1),\, r = s-1,\, a=-t-1,\, f = \frac{s(s+1)t(t+1)}{s+t},\, g = \frac{s^2(st+1)}{s+t}. \]
For any prime power $q$, there is a generalised quadrangle of order $(q, q^2)$. Using the parameters specified above, the collinearity graph of this generalised quadrangle has $n \approx q^4$ vertices, positive square energy $s^{+} \approx q^6$ and negative square energy $s^{-} \approx q^7$ (here we suppress all lower order terms). Therefore, $s^{+} \approx n^{6/4}$, $s^{-} \approx n^{7/4}$, and the result follows. 
 \end{proof}

 Another infinite family of graphs with $s^{-}/s^{+} \rightarrow \infty$ as $n\rightarrow \infty$ are the collinearity graphs of generalised quadrangles of order $(q^2, q^3)$ for a prime power $q$. Using the parameters given in the proof of Theorem~\ref{thm:gqex}, we find $n\approx q^7$, $s^{+} \approx q^{11}$ and $s^{-} \approx q^{12}$. While the ratio $s^{-}/s^{+} \approx n^{1/7}$ is not as good as the ratio given by the generalised quadrangles of order $(q, q^2)$, we note that the collinearity graphs of the generalised quadrangles of order $(q^2, q^3)$ have the following remarkable property: not only do we have $s^{+} = o(s^{-})$, but in fact $\mu_1^2 \approx q^{10}$, so we also have $\mu_1^2 = o(s^{+})$!

 We leave it as an open question to determine the maximum size of the ratio $s^{-}/s^{+}$. 

 \begin{ques}
     What is the maximum ratio that $s^{-}/s^{+}$ can achieve as $n\rightarrow \infty$? In particular, is there a family of graphs on $n$ vertices with $n\rightarrow \infty$ such that $s^{-}/s^{+} = \Theta(n)$? 
 \end{ques}

 \section{Squared spread of graphs}
 The \textit{spread} of  $G$  is defined to be $\mu_1 - \mu_n$~\cite{GHK2001}. Analogously, we define the \textit{squared spread} of a graph $G$ to be $s^{+} - s^{-}$. 

 The complete graph $K_n$ has squared spread $s^{+} - s^{-} = (n-1)^2 - (n-1) = (n-1)(n-2)$. We can readily deduce that this is nearly the best possible. 

 \begin{prop}
     For any  graph $G$ on $n$ vertices, $s^{+} - s^{-} = O(n^2)$. Furthermore, if Conjecture~\ref{con:elphick} is true, then $s^{+} - s^{-} \le (n-1)(n-2)$.
 \end{prop}

 \begin{proof}
 We have $s^{+} - s^{-} \le s^{+} \le 2m = O(n^2)$. If $G$ is disconnected, then $s^{+} - s^{-} \le s^{+} \le 2m \le (n-1)(n-2)$. Otherwise, if Conjecture~\ref{con:elphick} is true, then $s^{-} \ge n-1$, so $s^{+} - s^{-} \le (n-1)^2-(n-1) = (n-1)(n-2)$.
 \end{proof}

 We can ask the complementary question: how \textit{negative} can the squared spread be? Perhaps somewhat surprisingly, there are examples of graphs with very negative squared spread. 

 \begin{thm}\label{thm:negsqspd}
     For any $\epsilon > 0$, there is a graph $G$ on $n$ vertices with \[s^{+} - s^{-} = -\Omega(n^{2-\epsilon}).\] 
 \end{thm}

 \begin{proof}
We first find an infinite family of graphs where the squared spread $s^{+} - s^{-}$ is superlinearly negative. We use a family of strongly regular graphs defined by Taylor~\cite{Tay1, Tay2}. For any odd prime power $q$, there is a strongly regular graph $T_q$ on $n = q^3$ vertices with spectrum~\cite[Section 5.3]{Nik1} $k^1r^fs^g$ where: 
\[k = \frac12(q-1)(q^2+1),\, r = \frac12(q-1),\, s = -\frac12(q^2+1),\, f = (q-1)(q^2+1),\, g = q(q-1).\]
It is straightforward to compute from these parameters that the squared spread of $T_q$ is $s^{+} - s^{-} = -q^4/2 + q^3/2 - q^2/2 + q/2 = -\Theta(n^{4/3}).$

We now consider the $t$-blowups $T_q^{[t]}$ of the Taylor graphs. The \emph{$t$-blowup} of a graph $G$ is a graph $G^{[t]}$ obtained by replacing each vertex of $G$ with a set of $t$ independent vertices and replacing each edge of $G$ by a complete bipartite graph $K_{t, t}$. If $G$ has eigenvalues $\lambda_1, \ldots, \lambda_n$, then $G^{[t]}$ has eigenvalues $t\lambda_1, \ldots, t\lambda_n$ along with $(t-1)n$ additional 0s~\cite{Nik1}. 

The $t$-blowup $T_q^{[t]}$ is a graph on $q^3t$ vertices with squared spread $-\Theta(q^4t^2)$. Setting $t=q^a$ for some integer $a > 0$, the graph $T_q^{[t]}$ has $n = q^{3+a}$ vertices and squared spread $s^{+} - s^{-} = -\Theta(n^{(4+2a)/(3+a)}) = -\Theta(n^{2 - 2/(3+a)})$. Choosing $a$ sufficiently large gives an infinite family of graphs with squared spread $-\Omega(n^{2-\epsilon})$ for any $\epsilon > 0$. 
 \end{proof}

 The blowup construction used in the proof of Theorem~\ref{thm:negsqspd} works for any family of graphs $G$ with negative superlinear squared spread, so in particular it also works for the families of generalised quadrangles used in Section 3. 

 With our current knowledge, the situation for squared spread is therefore not quite symmetric between $s^{+}$ and $s^{-}$. It would be interesting to know if there is a family of graphs on $n$ vertices with $n\rightarrow \infty$ such that $s^{+} - s^{-} = -\Theta(n^{2})$. Unfortunately, graph blowups do not seem to yield such a family.  

 \begin{ques}
 Is there a family of graphs on $n$ vertices with $n \rightarrow \infty$ such that
 \[s^{+} - s^{-} = -\Theta(n^{2})?\]
 \end{ques}

\section{Square energies of maximal planar graphs}

Several authors have investigated upper bounds for the spectral radius of planar graphs. It is of interest to investigate similar bounds for the square energies of planar graphs, in part to see whether this reveals asymmetry between $s^+(G)$
and $s^-(G)$. 

Following Hong \cite{hong}, if Conjecture 1 is true then for any connected planar graph:

\[
\max{(s^+ , s^-)} \le 2m - n + 1 \le 5n - 11;
\]

since a planar graph has at most $3n - 6$ edges. 

A maximal planar graph, for $n \ge 3$,  has $m = 3(n - 2)$ and consequently has $s^+ + s^- = 2m = 6(n - 2)$. Ando and Lin \cite{ando} proved that:

\[
1 + \max{\left(\frac{s^+(G)}{s^-(G)} , \frac{s^-(G)}{s^+(G)}\right)} \le \chi(G). 
\]

Using the four colour theorem, it follows that for planar graphs $s^+(G) \le 3 s^-(G)$ and $s^-(G) \le 3s^+(G)$. $K_4$ provides an example of a connected maximal planar graph for which $s^+ = 3s^-$. We can however find no example of a connected maximal planar graph for which $s^- > s^+$ using the Wolfram database of graphs with up to 100 vertices. (There are planar graphs with $s^- > s^+$.)

This suggests the following question:

\begin{ques}
Is it true that for any connected maximal planar graph $G$, with $n \ge 3$, we have $s^+(G) \ge 3(n - 2) \mbox{ and } s^- \le 3(n - 2)?$
\end{ques}

In addition to this potential upper bound for $s^-(G)$ we can also prove the following lower bound as follows:

\[
4 \ge \chi(G) \ge 1 + \frac{s^+}{s^-} = \frac{2m}{s^-} = \frac{6(n - 2)}{s^-}.
\]

So $s^-(G) \ge 1.5(n - 2)$ and $s^+(G) \le 4.5(n - 2)$ for connected maximal planar graphs. 

To summarise, is it the case that for connected maximal planar graphs:

\[
3(n - 2) \le s^+(G) \le 4.5(n - 2) \mbox{  and  } 1.5(n - 2) \le s^-(G) \le 3(n - 2)?
\]

\section{Upper bounds for $s^-(G)$}

Conjecture~\ref{con:elphick} is equivalent to the statement that for any connected graph:

\[
s^-(G) \le 2m - n + 1, \mbox{  with equality for trees.}
\]

Elphick et al \cite{elphick16} proved that for all graphs $s^-(G) \le n^2/4$, with equality for complete regular bipartite graphs. This bound is not an upper bound for $s^+(G)$, so it provides an example of asymmetry. We note the bound $s^-(G) \le n^2/4$ can be strengthened as follows.

\begin{thm}\label{thm:mubsm}
For any graph $G$, 

\[
s^-(G) \le 2m - \frac{4m^2}{n^2}  \le \frac{n^2}{4}.
\]

\end{thm}

\begin{proof}
    The lower bound $\mu_1 \ge \frac{2m}{n}$ and the equality $s^{+}+s^{-}=2m$ give the upper bound 
    \[s^{-} = 2m - s^{+} \le 2m - \frac{4m^2}{n^2}.\]
    The inequality $2m - \frac{4m^2}{n^2}  \le \frac{n^2}{4}$ is equivalent to 
    \[\left(\frac{n}{2} - \frac{2m}{n}\right)^2 \ge 0.\]
\end{proof}

Theorem~\ref{thm:mubsm} is exact for regular complete bipartite graphs. These three upper bounds for $s^-(G)$ are illustrated in Figure 2, which is the same as Figure 1 with $n = 100$ except that we have replaced probability with edges on the x axis. The blue line is average square energies for random graphs; the green line is $s^- = 2m - n + 1$; the red line is $s^- = n^2/4$; and the brown line is $s^- = 2m - 4m^2/n^2 = d(n - d)$, where $d$ denotes average degree.

\begin{figure}[ht]
\centering
\begin{tikzpicture}

\begin{axis}[xmin=0, xmax = 4950, ymin=0, ymax= 10000, xlabel = {$m$}, ylabel = {negative square energy}]

\addplot+[color=blue] coordinates{
(0, 0)
(495, 464)
(990, 848)
(1485, 1149)
(1980, 1337)
(2475, 1437)
(2970, 1426)
(3465, 1308)
(3960, 1080)
(4455, 701)
(4950, 99)
};

\addplot+[color=red] coordinates{
(0, 2500)
(4950, 2500)
};

\addplot+[color=green] coordinates{
(99, 99)
(4950, 9801)
};

\addplot+[color=brown] coordinates{
(0, 0)
(200, 384)
(495, 892)
(990, 1588)
(1485, 2088)
(1980, 2392)
(2475, 2500)
(2970, 2412)
(3465, 2128)
(3960, 1648)
(4455, 972)
(4750, 475)
(4950, 99)
};

\end{axis}

\end{tikzpicture}

\caption{Upper bounds for $s^-$ and average values of $s^-$ with $n = 100$  }
\end{figure}
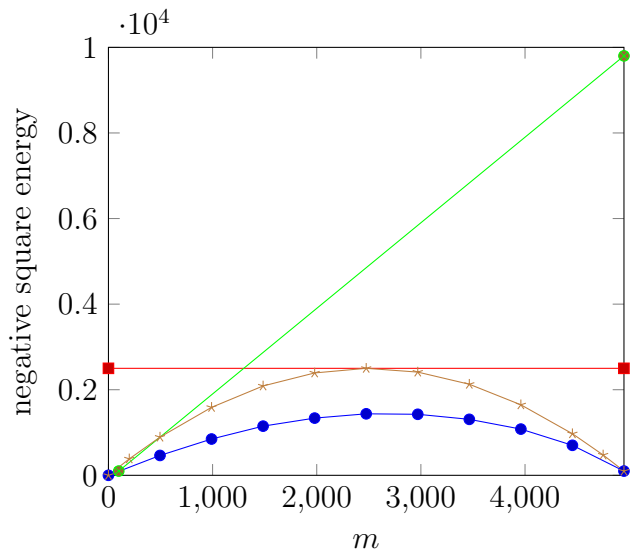

The bound from Theorem~\ref{thm:mubsm} reflects the symmetry around $p = 1/2$ that is observed in the experimental results for large enough random graphs.

 \section{Families of non-bipartite graphs with $s^+ = s^-$}

All bipartite graphs have equal positive and negative square energy. It was asked by Abiad \textit{et al.}~\cite{abiad} whether there are non-bipartite graphs with equal positive and negative square energy. The first author provided the example of the Kneser graph $K(6, 2)$. Recall that the Kneser graph $K(n, k)$ is a graph with vertex set $\binom{[n]}{k}$, the $k$-element subsets of the set $[n] := \{1, 2, \ldots, n\}$, and with an edge $AB \in K(n, k)$ if and only if $A\cap B = \emptyset$. 

 We extend the  example of $K(6, 2)$ to give an infinite family of Kneser graphs with equal positive and negative square energy. We then indicate how to extend this infinite family further to prove Theorem~\ref{thm:knes1}. 

 Recall that the eigenvalues of the Kneser graph $K(n, k)$ are
 \[\mu_i = (-1)^{i}\binom{n-k-i}{k-i},\]
 with $0\le i\le k$ and corresponding multiplicities
 \[\binom{n}{i} - \binom{n}{i-1}, \]
with the convention $\binom{n}{-1} = 0$.

 \begin{thm}\label{thm:2k+2thm}
 The Kneser graph $K(2k+2, k)$ has $s^{+} = s^{-}$ for every $k \ge 2$.  
 \end{thm}

 Note the example of $K(6, 2)$ is the special case of Theorem~\ref{thm:2k+2thm} when $k=2$. 

 We recall the following identity proved by Ruiz~\cite{Ruiz}.
 \begin{thm}[Ruiz]\label{thm:ruiz}
For all integers $n\ge 1$ and all real numbers $x$
\begin{equation}\label{eq:ruizfact}
\sum_{i=0}^{n}(-1)^{i}\binom{n}{i}(x-i)^n = n!.
\end{equation}
 \end{thm}

 We need the following corollary, also proved by Ruiz~\cite{Ruiz}, which is obtained by differentiating \eqref{eq:ruizfact} $j$ times.

 \begin{cor}[Ruiz]\label{cor:ruiz}
For all integers $n\ge 1$, for all real numbers $x$, and for an integer $j$ with $1\le j\le n$, 
\begin{equation}\label{eq:ruiz0}
\sum_{i=0}^n(-1)^i\binom{n}{i}(x-i)^{n-j} = 0.
\end{equation}
 \end{cor}

 \begin{proof}[Proof of Theorem~\ref{thm:2k+2thm}]
  If $n=2k+2$, then for $0\le i\le k$ we have
  \[\mu_i = (-1)^{i}\binom{k+2-i}{k-i} = (-1)^i\binom{k+2-i}{2}.\]
  Therefore, the positive and negative square energies are given by 
  \[s^{+} = \sum_{j=0}^{\lfloor{\frac{k}{2}\rfloor}}\binom{k+2-2j}{2}^2\left(\binom{2k+2}{2j} - \binom{2k+2}{2j-1}\right);\]
  \[s^{-} = \sum_{j=0}^{\lfloor{\frac{k-1}{2}\rfloor}}\binom{k+2 - (2j+1)}{2}^2\left(\binom{2k+2}{2j+1} - \binom{2k+2}{2j}\right)\]
  After rearranging and grouping by binomial coefficients of the form $\binom{2k+2}{j}$, in order to prove $s^{+} = s^{-}$, it suffices to show that
  \begin{equation}\label{eq:sqbinomeqn}
  \sum_{j=0}^k(-1)^j\binom{2k+2}{j}\left(\binom{k+2 - j}{2}^2 +\binom{k+1-j}{2}^2\right) = 0.
  \end{equation}
    Note that for any integer $a$ 
    \begin{align*}
 \binom{a}{2}^2+\binom{a-1}{2}^2 &= \frac{(a(a-1))^2 + ((a-1)(a-2))^2}{4}\\
 &=\frac{(a-1)^2(a^2+(a-2)^2)}{4}\\
 &=\frac{(a-1)^2(a^2-2a+2)}{2} = \frac12\left((a-1)^4+(a-1)^2\right).\\
 \end{align*}
  Therefore, after simplifying the sums of squares, \eqref{eq:sqbinomeqn} can be rewritten as 
  \[\sum_{j=0}^{k}(-1)^j\binom{2k+2}{j}\left((k+1-j)^4+(k+1-j)^2\right) = 0.\]
  Observe that if $j=k+1$, then $\binom{2k+2}{j}\left((k+1-j)^4+(k+1-j)^2\right) = 0$, and if $k+1 < j \le 2k+2$, then, \[(-1)^j\binom{2k+2}{j}\left((k+1-j)^4+(k+1-j)^2\right) = (-1)^{2k+2-j}\binom{2k+2}{2k+2-j}\left((j-k-1)^4 + (j-k-1)^2\right),\] so that
  \[\sum_{j=0}^{k}(-1)^j\binom{2k+2}{j}\left((k+1-j)^4+(k+1-j)^2\right) = \sum_{j=k+2}^{2k+2}(-1)^j\binom{2k+2}{j}\left((k+1-j)^4+(k+1-j)^2\right. )\]
  Hence, it suffices to show that
  \begin{equation}\label{eq:2k+2eqn}
\sum_{j=0}^{2k+2}(-1)^j\binom{2k+2}{j}\left((k+1-j)^4+(k+1-j)^2\right) = 0.
  \end{equation}
  But now Corollary~\ref{cor:ruiz} implies (with $x=k+1$) that
  \[\sum_{j=0}^{2k+2}(-1)^j\binom{2k+2}{j}(k+1-j)^4 = 0\]
  and 
  \[\sum_{j=0}^{2k+2}(-1)^j\binom{2k+2}{j}(k+1-j)^2 = 0,\]
  so \eqref{eq:2k+2eqn} holds and $s^{+}(K(2k+2, k)) = s^{-}(K(2k+2, k))$.
 \end{proof}

 Theorem~\ref{thm:knes1} can be proved similarly to Theorem~\ref{thm:2k+2thm}. The main technical difference is encapsulated in the following lemma. 

 \begin{lem}\label{lem:Plem}
 For a positive integer $j$, consider the function $P(a)=\binom{a}{2j}^2+\binom{a-1}{2j}^2$. Then, the polynomial $P(a)$ satisfies the following properties. 
  \begin{enumerate}
     \item $P(a) = 0$ for $a=1, \ldots, 2j-1$.
     \item $P(a) = Q(a-j)$, where $Q(a-j) = \sum_{m=0}^{4j}c_m(a-j)^m$ is a polynomial of degree $4j$ in the variable $a-j$ with $c_m = 0$ when $m$ is odd. 
 \end{enumerate}
 \end{lem}

 \begin{proof}[Proof of Lemma~\ref{lem:Plem}]
Note that for any integer $a$ 
    \begin{align*}
 \binom{a}{2j}^2+\binom{a-1}{2j}^2 &= \left(\frac{1}{(2j)!}\right)^2\left((a(a-1)\cdots(a-2j+1))^2+((a-1)(a-2)\cdots(a-2j))^2\right)\\
 &=\left(\frac{1}{(2j)!}\right)^2((a-1)\cdots(a-2j+1))^2(a^2+(a-2j)^2)\\
 &=\frac12\left(\frac{1}{(2j)!}\right)^2((a-1)\cdots(a-2j+1))^2((a-j)^2+j^2).\\
 \end{align*}
We need to prove two facts about $P(a)$:
 \begin{enumerate}
     \item $P(a) = 0$ for $a=1, \ldots, 2j-1$.
     \item $P(a) = Q(a-j)$, where $Q(a-j) = \sum_{m=0}^{4j}c_m(a-j)^m$ is a polynomial of degree $4j$ in the variable $a-j$ with $c_m = 0$ when $m$ is odd. 
 \end{enumerate}
 The first of these statements is clear since $a-1, a-2, \ldots, a-2j+1$ are factors of $P(a)$. The second follows from the following observation: for any integer $m$ with $1\le m \le j-1$, we have $(a - j + m)(a - j - m) = (a-j)^2 - m^2$. Hence, in the expression $(a-1)\cdots (a-2j+1)$, we may pair off the factors $(a-j+m)(a-j-m)$  for $1\le m \le j-1$, so that 
 \[P(a) = \frac12\left(\frac{1}{(2j)!}\right)^2(a-j)^2((a-j)^2+j^2)\prod_{m=1}^{j-1}((a-j)^2-m^2)^2.\]
 Expanding the product leaves a polynomial of degree $2j$ in the variable $x = (a - j)^2$, or equivalently a polynomial of degree $4j$ in the variable $y=a-j$ where the coefficients of odd power terms are zero. 
 \end{proof}

 The proof of Theorem~\ref{thm:knes1} follows the same line of argumentation as Theorem~\ref{thm:2k+2thm}, with the polynomial $Q(a-j)$ from Lemma~\ref{lem:Plem} used in place of the polynomial $\binom{a}{2}^2 + \binom{a-1}{2}^2 = \frac12\left((a-1)^4+(a-1)^2\right)$.

 \begin{proof}[Proof of Theorem~\ref{thm:knes1}]
  If $n=2k+2j$, then for $0\le i\le k$ we have
  \[\mu_i = (-1)^{i}\binom{k+2j-i}{k-i} = (-1)^i\binom{k+2j-i}{2j}.\]
  Therefore, the positive and negative square energies are given by 
  \[s^{+} = \sum_{i=0}^{\lfloor{\frac{k}{2}\rfloor}}\binom{k+2j-2i}{2j}^2\left(\binom{2k+2j}{2i} - \binom{2k+2j}{2i-1}\right);\]
  \[s^{-} = \sum_{j=0}^{\lfloor{\frac{k-1}{2}\rfloor}}\binom{k+2j - (2i+1)}{2j}^2\left(\binom{2k+2j}{2i+1} - \binom{2k+2j}{2i}\right)\]
  After rearranging and grouping by binomial coefficients of the form $\binom{2k+2}{i}$, in order to prove $s^{+} = s^{-}$, it suffices to show that
  \begin{equation}\label{eq:sqbinomeqn2j}
  \sum_{i=0}^k(-1)^i\binom{2k+2j}{i}\left(\binom{k+2j - i}{2j}^2 +\binom{k+(2j-1)-i}{2j}^2\right) = 0.
  \end{equation}
   By Lemma~\ref{lem:Plem}, the sum of squares $\binom{k+2j - i}{2j}^2 +\binom{k+(2j-1)-i}{2j}^2$ can be expressed as a polynomial $Q(k+j-i) = \sum_{m=0}^{4j}c_m(k+j-i)^m$ with $c_m = 0$ whenever $m$ is odd.  
  Therefore, \eqref{eq:sqbinomeqn2j} can be rewritten as 
  \[\sum_{i=0}^{k}(-1)^i\binom{2k+2j}{i}Q(k+j-i) = 0.\]
  Observe that if $k+1\le i\le k+2j-1$, then $\binom{2k+2}{i}Q(k+j-i) = 0$, by Property 1 of Lemma~\ref{lem:Plem}, and if $k+2j \le i \le 2k+2j$, then by Property 2 of Lemma!\ref{lem:Plem}, \[(-1)^i\binom{2k+2j}{i}Q(k+j-i)= (-1)^{2k+2j-i}\binom{2k+2j}{2k+2j-i}Q(k+j - (2k+2j-i)),\] so that
  \[\sum_{i=0}^{k}(-1)^i\binom{2k+2}{i}Q(k+j-i) = \sum_{i=k+2j}^{2k+2j}(-1)^i\binom{2k+2}{i}Q(k+j-i). \]
  Hence, it suffices to show that
  \begin{equation}\label{eq:2k+2jeqn}
\sum_{i=0}^{2k+2j}(-1)^i\binom{2k+2j}{i}Q(k+j-i) = 0.
  \end{equation}
  But now Corollary~\ref{cor:ruiz} implies \eqref{eq:2k+2jeqn} (with $x=k+j$) for each term $c_m(k+j-i)^m$ of $Q(k+j-i)$ (note that $k > j$ implies $2k+2j > 4j$) so that $s^{+}(K(2k+2j, k)) = s^{-}(K(2k+2j, k))$.
 \end{proof}


The number of positive, zero and negative eigenvalues of $K(n,k)$ are as follows \cite{elphick16}:

\[
n^+ = \binom{n - 1}{k} \mbox{  ;  } n^0 = 0 \mbox{  ;  } n^- = \binom{n - 1}{k - 1}.
\]
 
 Note that:

 \[
 s^+(K(n,k)) + s^-(K(n,k)) = 2m = \binom{n}{k} \binom{n - k}{k}, 
 \]
 so Theorems~\ref{thm:2k+2thm} and \ref{thm:knes1} imply that for $k > j \ge 1$,
 \[s^+(K(2k+2j, k)) = s^-(K(2k+2j, k)) = \frac12\binom{2k+2j}{k}\binom{k+2j}{k}.\]
 This raises the question:
 \begin{ques}
Are there closed-form formulae for $s^+$ and $s^-$ for any Kneser graph?
 \end{ques}
 
 \section{A weaker lower bound for square energies}
 
 Given the difficulty in proving Conjecture~\ref{con:elphick}, it is worthwhile to seek weaker lower bounds for square energies, such as the following.
 
 \begin{thm}
 
 For any connected graph $G$ with $n \ge 3$:
 \[
 \min{(s^-(G) , s^+(G))} \ge \sqrt{n}.
 \]
 \end{thm}
 
 \begin{proof}
 
 We know using Ando and Lin's \cite{ando} lower bound for the chromatic number and Hong's \cite{hong} upper bound for the spectral radius that:
 
 \[
 \min{(s^+(G) , s^-(G))} \ge \frac{2m}{\chi(G)} \ge \frac{2m}{1 + \mu} \ge \frac{2m}{1 +\sqrt{2m - n + 1}}.
 \]

 We therefore are seeking to prove that:
 
 \[
 \frac{2m}{1 + \sqrt{2m - n + 1}} \ge \sqrt{n}.
 \]
 
 This is equivalent to:
 
 \[
 (2m - \sqrt{n})^2 \ge (2m - n + 1)n;
 \]
 
 which simplifies to:
 
 \begin{equation}\label{eqn:4m^2}
 4m^2 - 2m(2\sqrt{n} + n) + n^2 \ge 0.
 \end{equation}
 
 Since $G$ is connected, $m \ge n - 1$ and if $m = n - 1$ then $G$ is a tree with $s^+ = s^- = n - 1 \ge \sqrt{n}$ for $n \ge 3$. We can therefore assume $m \ge n$.  
 
 If $m = n$ then \eqref{eqn:4m^2} becomes:
 
 \[
 4n^2 - 4n\sqrt{n} - 2n^2 + n^2 = 3n^2 - 4n\sqrt{n} \ge 0 \mbox{ for } n \ge 3.
 \]
 
 The left hand side of \eqref{eqn:4m^2} monotonically increases with $m$, so if $m > n$ then \eqref{eqn:4m^2} remains true. This completes the proof. \end{proof}

 \section{Disconnected graphs}
 
Conjecture~\ref{con:elphick} is stated for connected graphs because it is false, for example, for $2K_{n/2}$. We can however prove the following result for regular graphs. 

\begin{thm}

   Let $G$ be a $d-$regular disconnected graph with no complete component. Then:

   \[
   \min{(s^+(G) , s^-(G))} \ge n - 1.
   \]
\end{thm}

\begin{proof}
\emph{Case 1 : No component of $G$ is an odd cycle}

We follow the approach in \cite{elphick16}. Brooks' Theorem from 1941 proves that any connected graph other than a complete graph or an odd cycle has $\chi(G) \le \Delta$, where $\Delta$ is the largest vertex degree. This theorem can therefore be applied to each component of a graph, provided no component is complete or an odd cycle. As discussed in the Introduction, Ando and Lin proved that $s^+(G) \ge 2m/\chi(G)$ and $s^-(G) \ge 2m/\chi(G).$ Therefore:

\[
\min(s^+(G) , s^-(G)) \ge \frac{2m}{\chi(G)} \ge \frac{2m}{\Delta} = \frac{2m}{d} = n.
\]

\emph{Case 2: A component of $G$ is an odd cycle}

Since $G$ is 2-regular it follows that all components are cycles. The even cycles have $s^+ = s^- = m = n$. Abiad \emph{et al.} \cite{abiad} proved that all odd cycles satisfy Conjecture~\ref{con:elphick}. Therefore $G$ satisfies Conjecture~\ref{con:elphick}.

\end{proof}

This result raises questions. Will a proof of Conjecture~\ref{con:elphick} need to have connectedness at its heart, or can Conjecture~\ref{con:elphick} be generalised to any graph subject to a limited number of exclusions, which include isolated vertices and complete components? Can Conjecture~\ref{con:elphick} be generalised for all graphs (apart from isolated vertices) for $s^+$ but not for $s^-$?

Let $G$ be any graph for which $s^+ = n - \epsilon$ or $s^- = n - \epsilon$, where $0 < \epsilon \le 1$. It then follows that if we take a sufficiently large number of copies of $G$, then this disconnected graph will not satisfy Conjecture~\ref{con:elphick} for $s^+$ or $s^-$ respectively. There are numerous non-bipartite connected graphs which have $s^+ < n$ or $s^- < n$.

Consequently, it seems implausible that Conjecture~\ref{con:elphick} could be proved for disconnected irregular graphs for $s^+$ or $s^-$ with only a limited number of exclusions. This provides evidence that connectedness will be central to a proof of Conjecture~\ref{con:elphick} for $s^+$ and $s^-$, which provides another example of symmetry between $s^+$ and $s^-$. 

If connectedness is central to  a proof of Conjecture~\ref{con:elphick}, then the following weaker conjecture may be more tractable. Let $n^+, n^0$ and $n^-$ denote the number of positive, zero and negative eigenvalues respectively, where $n^+ + n^0 + n^- = n$. 

\begin{con}
For any graph $G$ with inertia $(n^+, n^0, n^-)$:

\[
\min{(s^+(G) , s^-(G))} \ge \max{(n^+ ,n^-)}.
\]
    
\end{con}

\section{Initial ideas for a proof of Conjecture~\ref{con:elphick}}

This paper has extended the early symmetric results on square energies to consider asymmetry. For example, we have determined $s^{+}$ and $s^{-}$ a.s. for almost all graphs. The results in Section 9 seem to be most relevant to a potential proof of Conjecture~\ref{con:elphick}. 

As discussed in \cite{elphick16}, Conjecture~\ref{con:elphick} relates to irreducible, symmetric, binary, zero-trace matrices. The difficulty of proving the conjecture appears to be due to the need for graph connectedness, which is equivalent to matrix irreducibility, to be central to a proof. So perhaps a proof can be expected to use a spectral resolution of $A$ as follows, where $v_1, ... ,v_n$ are the column unit eigenvectors corresponding to $\mu_1, ..., \mu_n$ and $v_i^T$ denotes the transpose of $v_i$:

\[
A= \sum_{i=1}^n \mu_i v_i v_i^T; \mbox{     } B= \sum_{\mu_i > 0} \mu_i v_i v_i^T; \mbox{   } C = \sum_{\mu_i < 0} (-\mu_i) v_i v_i^T.
\]

It is then the case that:

\[
A = B - C; s^+ = Tr(B^2); s^- = Tr(C^2) \mbox{  and  } PAP^{-1} \not = \begin{pmatrix} E & F \\ 0 & G \end{pmatrix},
\]

where $P$ is a permutation matrix and $E$ and $G$ are square matrices of size $\ge 1$. These formulae for $s^+$ and $s^-$ are used in \cite{wocjan13}, where it is noted that $B$ and $C$ are both positive semidefinite.

\begin{ques}\label{ques:irrques}
    Does the irreducibility of $A$ imply that $B$ and $C$ are both irreducible?
\end{ques}

If the answer to Question~\ref{ques:irrques} is yes, then a proof could use that:

\[
PBP^{-1} \not = \begin{pmatrix} E & F \\ 0 & G \end{pmatrix} \mbox{   and   }PCP^{-1} \not = \begin{pmatrix} E & F \\ 0 & G \end{pmatrix}\],

or an alternative property of irreducible, positive semidefinite matrices. 

Perhaps it is possible to generalise the conjecture, for example by replacing $A$ with a weighted adjacency matrix with positive weights. Such a generalisation would suggest that the conjecture may relate not only to graphs but could be applicable to broader classes of matrices. 

\section*{Acknowledgements}

The authors thank the reviewers for their careful and helpful comments on the paper.

\end{document}